\documentclass[11pt]{amsart}

\usepackage{mathptmx} 
\usepackage[scaled=0.90]{helvet} 
\usepackage{courier} 
\normalfont
\usepackage[T1]{fontenc}
\usepackage{color}

\usepackage{hyperref}
\usepackage{amsmath}
\usepackage{amsthm}
\usepackage{amsfonts}
\usepackage{amssymb}

\usepackage{graphicx}
\DeclareGraphicsExtensions{.eps}

\newcommand{\eps}{\varepsilon}

\usepackage{amssymb,amsmath,amsthm}
\usepackage{esint}
\usepackage{mathrsfs}
\usepackage{comment}
\usepackage{relsize}
\usepackage{bbm,dsfont}
\usepackage[margin=0.96in]{geometry}

\usepackage{amssymb,amsmath,amsthm}
\usepackage{esint}
\usepackage{mathrsfs}
\usepackage{comment}
\usepackage{relsize}
\usepackage{bbm,dsfont}
\usepackage{color}

\newtheorem{theorem}{Theorem}

\newtheorem{lemma}{Lemma}
\newtheorem{corollary}{Corollary}

\newcommand{\bE}{\mathbb{E}}
\newcommand{\cL}{\mathcal{L}}

\newcommand{\cF}{\mathcal{F}}
\usepackage{graphicx}

\begin{document}
\title{From discrete flow of Beckner to continuous flow of Janson in complex hypercontractivity}
\author[P.~Ivanisvili, A.~Volberg]{P.~Ivanisvili,  A.~Volberg}
\thanks{A.~Volberg is partially supported by the NSF DMS-1600065.  This paper is  based upon work supported by the National Science Foundation under Grant No. DMS-1440140 while two of the authors were in residence at the Mathematical Sciences Research Institute in Berkeley, California, during the Spring and Fall  2017 semester. }
\address{Department of Mathematics, Princeton University; MSRI; UC Irvine, CA, USA}
\email{paata.ivanisvili@princeton.edu \textrm{(P.\ Ivanisvili)}}


\address{Department of Mathematics, Michigan State University, East Lansing, MI 48823, USA}
\email{volberg@math.msu.edu \textrm{(A.\ Volberg)}}
\makeatletter
\@namedef{subjclassname@2010}{
  \textup{2010} Mathematics Subject Classification}
\makeatother
\begin{abstract} 
We show how Beckner's montonicity result on Hamming cube easily implies the monotonicity of a  flow introduced by Janson in Hausdorff--Young inequality
\end{abstract}
\maketitle






%
%



\date{}

\setcounter{equation}{0}
\setcounter{theorem}{0}

\section{Complex hypercontractivity: discrete and continuous monotonicity }
\label{Intro}

For  $1 \leq p \leq q  <\infty$ and $|z|\leq 1$ with  $z \in \mathbb{C}$ the complex hypercontractivity 
\begin{align*}
 (\mathbb{E} | T_{z} f|^{q})^{1/q} \leq  (\mathbb{E} |f|^{p})^{1/p} \quad \text{for all} \quad f :\{-1,1\}^{n} \to \mathbb{C} 
\end{align*}
where  $T_{z} f(x) = \sum_{S \subset \{1,\ldots, n\}} z^{|S|} \hat{f}(S) W_{S}(x)$, $W_{S}(x)=\prod_{j\in S}x_{j}$, $x=(x_{1}, \ldots, x_{n})\in \{-1,1\}^{n}$  is equivalent to its two-point inequality 
\begin{align}\label{basis}
\left(\frac{|a+zb|^{q}+|a-zb|^{q}}{2}\right)^{1/q} \leq \left(\frac{|a+b|^{p}+|a-b|^{p}}{2}\right)^{1/p} \quad \text{for all} \quad a,b \in \mathbb{C}, 
\end{align}
which is conjectured to be equivalent to its infinitesimal form, i.e.,  
\begin{align*}
(q-2) (\Re w z)^{2} + |w z|^{2} \leq (p-2) (\Re w)^{2} + |w|^{2}, \quad \forall w \in \mathbb{C}, 
\end{align*}
(the only open case is when $2 < p <q < 3$ and its dual~\cite{Weissler}). 

The proof of hypercontractivity goes as follows
\begin{align}
(\mathbb{E} |T_{z} f|^{q})^{p/q}  &= \nonumber\\
 &=\left( \mathbb{E}_{n-1} \mathbb{E}^{1} |T_{z} f|^{q}\right)^{p/q} \; \stackrel{(\ref{basis})}{\leq}  
\left( \mathbb{E}_{n-1}  \left( \mathbb{E}^{1}|T^{1}_{z}f|^{p} \right)^{q/p}\right)^{p/q} \; \stackrel{\text{minkowski}}{\leq}  \mathbb{E}^{1}  \left( \mathbb{E}_{n-1}|T^{1}_{z}f|^{q} \right)^{p/q}  \label{bili}\\
&\stackrel{\text{induction}}{\leq} \mathbb{E} |f|^{p}.\nonumber
\end{align}
Here $\mathbb{E}_{n-1}, \mathbb{E}^{1}$ averages in the last $n-1$ and the first variable correspondingly; $T_{z}^{1}$ {\em removes} one $z$ from the $x_{1}$ variable in the formula for $T_{z}$. In general we define $T_{z}^{k} W_{S}(x) = z^{|S\cap\{k+1,\ldots, n\}|}W_{S}(x)$ for $1\leq  k< n$, i.e., $T_{z}^{k}$ multiplies only variables $x_{k+1}, \ldots, x_{n}$ by $z$. We set  $T_{z}^{0}=T_{z}$, $T_{z}^{n}W_{S}(x)=W_{S}(x)$, and we extend $T_{z}^{k}$ linearly.  Introduce the symmetric functions 
\begin{align*}
 \varphi_{\ell} (x_{1}, \ldots, x_{n}) = \ell! \sum_{1\leq m_{1}<m_{2}<\ldots<m_{\ell}\leq n}x_{m_{1}}\ldots x_{m_{\ell}}.
\end{align*}
Clearly   
\begin{align*}
T_{z}^{k} \varphi_{\ell}  (x_{1}, \ldots, x_{n}) = \varphi_{\ell} (x_{1}, \ldots, x_{k}, z x_{k+1}, \ldots, z x_{n}), \quad \text{for} \quad 1\leq k <n.  
\end{align*}
Inducting (\ref{bili}) implies that the following discrete map is monotonically increasing in $k$ (Beckner~\cite{B}): 
\begin{align}\label{izrdeba}
k \to \mathbb{E}^{k} (\mathbb{E}_{n-k} | T_{z}^{k} f|^{q})^{p/q}, \quad 0 \leq k \leq n.
\end{align}
We will show that (\ref{izrdeba}) implies that the following map  is increasing:
\begin{align}\label{mon2}
s \to \int_{\mathbb{R} }\left( \int_{\mathbb{R}} \left| \int_{\mathbb{R}} \int_{\mathbb{R}} g((u+iv)\sqrt{s} + z (x+iy)\sqrt{1-s}) d\gamma(v) d\gamma(y) \right|^{q} d\gamma(x)\right)^{p/q}d\gamma(u), \quad 0\leq s \leq 1,
\end{align}
where $d\gamma(x)=\frac{e^{-x^{2}/2}}{\sqrt{2\pi}}dx$ and $g$ is a polynomial. The monotonicity of (\ref{mon2}) was proved in \cite{SJ} and later in \cite{Hu} by direct  but tedious differentiation that takes  a little bit of time to achieve though it is nontrivial to guess. We show that it follows from Beckner's paper \cite{B} directly.

\section{From discrete monotonicity to continuous monotonicity}
\label{DtoC}


We will need the following lemma of Beckner~\cite{B}.
\begin{lemma}[Beckner~\cite{B}]
\label{BeL}
Let $x\in \{-1,1\}^{n}$.  We have
$$
\varphi_{\ell} \left( \frac{x_1}{\sqrt{n}},\dots, \frac{x_n}{\sqrt{n}}\right) = H_\ell\left(\frac{x_1+ \dots + x_n}{\sqrt{n}}\right)  + \frac1{n}  \sum_{r=1}^{[\ell/2]} a_{r, \ell}H_{\ell-2r}\left( \frac{x_1+\dots+x_n}{\sqrt{n}}\right)\,,
$$
where the coefficients $a_{r, \ell}$ are bounded with respect to $n$ for each fixed $\ell$, and $H_m$ are Hermite polynomials:
$$
H_m(x):= \int_{\mathbb{R}} (x+iy)^m d\gamma(y)\,.
$$
\end{lemma}

\bigskip

Now for a fixed $T>0$ consider the scaled Hamming ball
$$
B_{n}(T):=\left\{x\in \{-1,1\}^n\, :\,  \left|\frac{x_1+\dots+ x_n}{\sqrt{n}}\right| \le T\right\}\,.
$$
We claim 
\begin{lemma}
\label{Tz}
Let $k \asymp n$,  $n \to \infty$. Uniformly for $x=(x',x'')\in \{-1,1\}^n$ such that $x'=(x_1,\dots, x_k)\in B_{k}(T)$ and $x''=(x_{k+1}, \dots, x_n) \in B_{n-k}(S)$  we have
\begin{align*}
&\mathbb{E}^{y}\left(\frac{x_{1}+\ldots+x_{k}}{\sqrt{n}}+i \frac{y_{1}+\ldots+y_{k}}{\sqrt{n}}+z \left[\frac{x_{k+1}+\ldots+x_{n}}{\sqrt{n}}+i \frac{y_{k+1}+\ldots+y_{n}}{\sqrt{n}} \right]\right)^{L} =\\
&\varphi_{L}\left( \frac{x_{1}}{\sqrt{n}}, \ldots,\frac{x_{k}}{\sqrt{n}}, z\frac{x_{k+1}}{\sqrt{n}}, \ldots, z\frac{x_{n}}{\sqrt{n}}\right) 
+O_{L, T, S}\left(\frac1{\sqrt{n}}\right)
\end{align*}
where $\mathbb{E}^{y}$ takes the average in the variable  $y = (y_{1}, \ldots, y_{n})\in \{-1,1\}^{n}$.
\end{lemma}

\begin{proof}
First we explain that the case of an arbitrary $z$ follows from the case $z=0$. Indeed, 
\begin{align*}
&\mathbb{E}^{y}\left(\frac{x_{1}+\ldots+x_{k}}{\sqrt{n}}+i \frac{y_{1}+\ldots+y_{k}}{\sqrt{n}}+z \left[\frac{x_{k+1}+\ldots+x_{n}}{\sqrt{n}}+i \frac{y_{k+1}+\ldots+y_{n}}{\sqrt{n}} \right]\right)^{L}=\\
&\sum_{m=0}^{L}\binom{L}{m} \mathbb{E}^{y_{1}, \ldots, y_{k}}\left( \frac{x_{1}+\ldots+x_{k}+i(y_{1}+\ldots+y_{k})}{\sqrt{n}} \right)^{L-m} \mathbb{E}^{y_{k+1}, \ldots, y_{n}}\left(\frac{x_{k+1}+\ldots+x_{n}+i(y_{k+1}+\ldots+y_{n})}{\sqrt{n}}\right)^{m}z^{m}.
\end{align*}
Next, using the validity of the lemma when  $z=0$, we can write the latter sum as follows 
\begin{align*}
\sum_{m=0}^{L}\binom{L}{m}\varphi_{L-m}\left(\frac{x_{1}}{\sqrt{n}}, \ldots, \frac{x_{k}}{\sqrt{n}}\right)\varphi_{m}\left(\frac{x_{k+1}}{\sqrt{n}}, \ldots, \frac{x_{n}}{\sqrt{n}}\right)z^{m} + O_{L,T,S,z}\left(\frac{1}{\sqrt{n}}\right).
\end{align*}
Finally,  the following identity 
\begin{align*}
\varphi_{L}(x_{1}, \ldots, x_{k}, z x_{k+1}, \ldots, zx_{n}) = \sum_{m=0}^{L}\binom{L}{m}\varphi_{L-m}\left(x_{1}, \ldots, x_{n}\right)\varphi_{m}\left(x_{k+1}, \ldots, x_{n}\right)z^{m}
\end{align*}
finishes the proof. Thus it remains to prove the lemma only when $z=0$. We say that $P=P_{\lambda}$ partitions a natural number $\lambda$ if $P_{\lambda}=(\lambda_{1}, \ldots, \lambda_{m})$ for some $m\geq 1$,  where $\lambda_{j}$ are natural numbers, $\lambda_{j}\geq \lambda_{j+1}$ and $\lambda_{1}+\ldots+\lambda_{m}=\lambda$. By $|P_{\lambda}|$ we denote the {\em width} of the partition $P_{\lambda}$, i.e., in this case $|P_{\lambda}|=m$. Let $(X_{1}, \ldots, X_{k}) \in \mathbb{C}^{k}$. If $P_{\lambda}$ is  a partition of $\lambda$, and $\lambda \leq k$ then  by $M_{P_{\lambda}}(X_{1}, \ldots, X_{k})$ we denote the symmetric polynomials. For example 
\begin{align*}
&M_{(1,1,1)}(X_{1},X_{2},X_{3})=X_{1}X_{2}X_{3};\\
&M_{(1,1)}(X_{1},X_{2},X_{3})=X_{1}X_{2}+X_{2}X_{3}+X_{1}X_{3};\\
&M_{(3,1)}(X_{1}, X_{2}, X_{3})=X_{1}^{3}X_{2}+X_{1}^{3}X_{3}+X_{2}^{3}X_{1}+X_{2}^{3}X_{3}+X_{3}^{3}X_{1}+X_{3}^{3}X_{2}.
\end{align*}
Next, consider the multinomial expansion
\begin{align*}
n^{-L/2}\mathbb{E}^{y}\left(x_{1}+iy_{1}+\ldots+x_{k}+iy_{k}\right)^{L}=n^{-L/2}\mathbb{E}^{y}\sum_{|r|=L}\binom{L}{r} X^{r}
\end{align*}
where $r=(r_{1}, \ldots, r_{k})$ is the multiindex with nonnegative integers $r_{j}$, $|r|=r_{1}+\ldots+r_{k}$, $\binom{L}{r}:=\binom{L}{r_{1}, \ldots, r_{k}}$, and 
\begin{align*}
X^{r} := (x_{1}+iy_{1})^{r_{1}}\cdots(x_{k}+iy_{k})^{r_{k}}.
\end{align*}
Using the notations with partition numbers and symmetric polynomials we can write
\begin{align}\label{Volberg}
n^{-L/2}\mathbb{E}^{y}\sum_{|r|=L}\binom{L}{r} X^{r} = \sum_{w=1}^{L}\; \sum_{\substack{P_{L}\, \text{partitions}\, L \\ |P_{L}|=w}} \binom{L}{P_{L}} n^{-L/2}\mathbb{E}^{y} M_{P_{L}}(x_{1}+iy_{1}, \ldots, x_{k}+iy_{k})
\end{align}
When $w=L$ the only possible partition of $L$ with width $L$ is of course $P_{L}=\{\underbrace{1,\ldots, 1}_{L}\}$. Clearly in this case 
\begin{align*}
L!\,  n^{-L/2} \mathbb{E}^{y} M_{P_{L}}(x_{1}+iy_{1},\cdots,x_{k}+iy_{k}) = \varphi_{L}\left(\frac{x_{1}}{\sqrt{n}}, \ldots, \frac{x_{k}}{\sqrt{n}}\right)
\end{align*}
Let us show that all other terms $n^{-L/2}\mathbb{E}^{y} M_{P_{L}}(x_{1}+iy_{1}, \ldots, x_{k}+iy_{k})$ of the double sum (\ref{Volberg}) are of order $\frac{1}{\sqrt{n}}$ when $1\leq w\leq L-1$. Notice that 
\begin{equation}
\label{sokr}
(x_j +iy_j)^2=  2ix_jy_j,\, \,\,(x_j +iy_j)^4= -4, \,\,\,(x_j +iy_j)^3= 2i(y_j +ix_j),
\end{equation}
Using this we can reduce  each element of $P_{L}$ with $|P_{L}|=w$ to the following values $m_j=1, 2,  0$, and denote it by $P'_{L}$.  We do not care about numerical coefficients that may appear after the reduction. This is because  we cannot pick up a coefficient bigger than $4^L$, and constants depending only on $L$ are fine with us. If $P'_{L}$ contains $2$ then clearly $\mathbb{E}^{y} M_{P_{L}}=0$. Assume all elements of $P'_{L}$ are zero.  Then  there can be only at most $L/4$ zeros. It means that the total contribution from  $M_{P_{L}}$ can be  at most $C(L) n^{L/4}$ (we remind that $k \asymp n$) which after multiplication by $n^{-L/2}$ will go to zero. Finally consider the case when $P'_{L}$ contains $K$, $1\leq K\leq L-1$,  number of $1$'s. Then clearly $n^{-L/2}\mathbb{E}^{y} M_{P_{L}} = k^{-L/2}C(L,k/n) \varphi_{K}(x_{1}, \ldots, x_{k})$. By Beckner's Lemma~\ref{BeL} we have 
$$
\frac1{k^{(L-K)/2}} \varphi_{K}(x_1/\sqrt{k},\dots, x_k/\sqrt{k})= \frac1{k^{(L-K)/2}} \sum_{\nu=0}^K a_{\nu, K}H_\nu\left(\frac{x_1+\dots+x_k}{\sqrt{k}}\right)\,,
$$
where $H_\nu$ are Hermite polynomials.   Obviously, as $K\le L-1$ coefficients $a_{\nu, K}$ are bounded by certain $C'(L)$. Since $(x_{1}, \ldots, x_{k})\in B_{k}(T)$ the latter sum is of order $O_{L,T}(1/\sqrt{n})$. The lemma is proved. 
\end{proof}

We will see now
that not only Lemma \ref{Tz} holds, but moreover, that 

\begin{lemma}
\label{InAv}
If $k=k(n)$ is such that  $\lim_{n \to \infty}k/n = s \in [0,1]$, $a_{\ell}\in \mathbb{C}$ for $\ell=0,\ldots, L$ and $z \in \mathbb{C}$,  then
\begin{align*}
&\lim_{n\to \infty} \bigg[ \mathbb{E}_{x}^{k}\left(\mathbb{E}^{x}_{n-k} \left|\mathbb{E}^{y}\sum_{\ell=0}^{L}a_{\ell} \left(\frac{x_{1}+\ldots+x_{k}}{\sqrt{n}}+i \frac{y_{1}+\ldots+y_{k}}{\sqrt{n}}+z \left[\frac{x_{k+1}+\ldots+x_{n}}{\sqrt{n}}+i \frac{y_{k+1}+\ldots+y_{n}}{\sqrt{n}} \right]\right)^{\ell} \right|^{q} \right)^{p/q} -\\
&\mathbb{E}^{k} \left(\mathbb{E}_{n-k} \left| T_{z}^{k} \left(\sum_{\ell=0}^{L} a_{\ell}\varphi_{\ell}\left(\frac{x_{1}}{\sqrt{n}}, \ldots, \frac{x_{n}}{\sqrt{n}}\right)\right)\right|^{q}\right)^{p/q}\bigg] =0.
\end{align*}
\end{lemma}

\begin{proof}

We have the sequences of random variables
\begin{align*}
&\xi_n = \frac{x_1+\dots +x_k}{\sqrt{n}} = \sqrt{s_n} \frac{x_1+\dots+ x_k}{\sqrt{k}}, \, s_n= k/n;\\
&\eta_n= \frac{x_{k+1}+\dots +x_n}{\sqrt{n}} =\sqrt{1- s_n} \frac{x_{k+1}+\dots +x_n}{\sqrt{n-k}};\\
&\zeta_n=  \frac{y_1+\dots +y_k}{\sqrt{n}} = \sqrt{s_n} \frac{y_1+\dots +y_k}{\sqrt{k}}, \, s_n= k/n;\\
&\tau_n=  \frac{y_{k+1}+\dots +y_n}{\sqrt{n}} = \sqrt{1-s_n} \frac{y_{k+1}+\dots +y_n}{\sqrt{n-k}}\,.
\end{align*}
For a fixed $n$ random variables $(\xi_n, \eta_n, \zeta_n, \tau_n)$ are pairwise independent. All these random variables are {\it uniformly} sub-gaussian.

Now $z$ is a complex number, and we consider $g(\xi_n+ i \zeta_n + z(\eta_n+i \tau_n))$ and $f_{n}(x)$
where
\begin{align*}
&g(w) := \sum_{\ell=0}^{L}a_{\ell}\,w^{\ell};\\
&f_{n}(x_{1}, \ldots, x_{n}) := \sum_{\ell=0}^{L} a_{\ell}\varphi_{\ell}\left(\frac{x_{1}}{\sqrt{n}}, \ldots, \frac{x_{n}}{\sqrt{n}}\right). 
\end{align*}
Obviously with $A, B<\infty$ (depending on $L$, $z$, and $\{a_\ell\}_{\ell=0}^{L}$ only)
\begin{equation}
\label{g1}
|g(\xi_n+ i \zeta_n + z(\eta_n+i \tau_n))| \le A e^{B(|\xi_n|+|\zeta_n|+|\eta_n|+ |\tau_n|)}\,.
\end{equation}
So if we consider (for $\beta\ge 1, \alpha\le 1$)
\begin{align*}
&\bE_\xi( \bE_\eta |\bE_{\zeta, \tau} g(\xi_n+ i \zeta_n + z(\eta_n+i \tau_n))|^{\beta})^\alpha\le \bE_\xi( \bE_\eta \bE_{\zeta, \tau} |g(\xi_n+ i \zeta_n + z(\eta_n+i \tau_n)|^{\beta})^\alpha\le \\
&A\bE_\xi( \bE_\eta \bE_{\zeta, \tau}e^{\beta B(|\xi_n|+|\zeta_n|+|\eta_n|+ |\tau_n|}))^\alpha= A\bE_\xi e^{\alpha\beta B(|\xi_n|}( \bE_\eta \bE_{\zeta, \tau}e^{\beta B|\zeta_n|}e^{\beta B|\eta_n|}e^{\beta B |\tau_n|}))^\alpha \le\\
&A\bE_\xi e^{\alpha\beta B|\xi_n|} (\bE_\eta e^{\beta B|\eta_n|} \bE_\zeta e^{\beta B|\zeta_n|}\bE_\tau e^{\beta B|\tau_n|})^\alpha\le A\bE_\xi e^{\alpha\beta B|\xi_n|} Const\,,
\end{align*}
where $Const$ depends only on sub-gaussian norms, but not on $n$. Moreover, as $\xi_n$ are also uniformly sub-gaussian, the above calculation shows that
\begin{equation}
\label{cut1}
\bE_\xi(\bE_\eta |\bE_{\zeta, \tau} g(\xi_n+ i \zeta_n + z(\eta_n+i \tau_n))|^{\beta})^\alpha= \bE_\xi {\bf 1}_{|\xi_n|\le T}(\bE_\eta |\bE_{\zeta, \tau} g(\xi_n+ i \zeta_n + z(\eta_n+i \tau_n))|^{\beta})^\alpha + \eps_T,
\end{equation}
where $\eps_T$ does not depend on $n$ and
$$
\lim_{T\to \infty} \eps_T=0\,.
$$
In the right hand side of \eqref{cut1} we can now  truncate  $\eta_n$. In fact, rewriting \eqref{g1} as
\begin{equation}
\label{g1}
{\bf 1}_{|\xi_n|\le T}|g(\xi_n+ i \zeta_n + z(\eta_n+i \tau_n))| \le A e^{BT}e^{B(|\zeta_n|+|\eta_n|+ |\tau_n|)}\,.
\end{equation}
we can continue \eqref{cut1}:
\begin{align}
\label{cut2}
&\bE_\xi {\bf 1}_{|\xi_n|\le T}(\bE_\eta |\bE_{\zeta, \tau} g(\xi_n+ i \zeta_n + z(\eta_n+i \tau_n))|^{\beta})^\alpha =\\
& \bE_\xi {\bf 1}_{|\xi_n|\le T}(\bE_\eta  {\bf 1}_{|\eta_n|\le S(T)}|\bE_{\zeta, \tau} g(\xi_n+ i \zeta_n + z(\eta_n+i \tau_n))|^{\beta})^\alpha + Ae^{BT} \delta_S\,,
\end{align}
where $\delta_S$ does not depend on $n$ and
$$
\lim_{S\to \infty} \delta_S=0\,.
$$

\medskip

Thus, choosing independently of $n$ a very large $T$ first and then very large $S$ to make $\eps_T$ and  then $Ae^{BT} \delta_S$ smaller than a given positive
number $\nu$,  we are now under the assumptions of   Lemma \ref{Tz}. 

Therefore, Lemma \ref{Tz}  claims that we have uniform in the values of $\xi_n\in [-T, T]$, $\eta_n \in [-S(T), S(T)]$ closeness of 
$\bE_{\zeta, \tau}g(\xi_n+ i \zeta_n + z(\eta_n+i \tau_n))$ and $T_z^k f_{n}$. Their difference does not exceed $C(L, \{a_\ell\}_{\ell=1}^L, T, S(T))/\sqrt{n}$, as Lemma \ref{Tz} shows.

In particular we proved that the difference between 
$$
P:=\bE_\xi{\bf 1}_{|\xi_n|\le T}( \bE_\eta {\bf 1}_{|\eta_n|\le S} |\bE_{\zeta, \tau} g(\xi_n+ i \zeta_n + z(\eta_n+i \tau_n))|^{\beta})^\alpha
$$
and 
$$
Q:=\bE_\xi{\bf 1}_{|\xi_n|\le T}\left( \bE_\eta {\bf 1}_{|\eta_n|\le S}\left|T_z^k f_{n}\right|
^\beta\right)^\alpha\,
$$
tends to zero as $n \to \infty$.

By the choice of large $T$ and then large $S=S(T)$ we made  the expression $P$ above as close as we wish to
$$
P':=\bE_\xi( \bE_\eta  |\bE_{\zeta, \tau} g(\xi_n+ i \zeta_n + z(\eta_n+i \tau_n))|^{\beta})^\alpha\,.
$$

We are left to see that by the choice of large $T$ and then large $S=S(T)$ we made  the expression $Q$ above as close as we wish to
$$
Q':=\bE_\xi( \bE_\eta |T_z^k f_{n}|
^\beta)^\alpha\,.
$$
Write $\bE_\xi( \bE_\eta |T_z^k f_{n}|^\beta)^\alpha= \bE_\xi( \bE_\eta | T_z^k(\sum_{\ell=0}^{L} a_{\ell} \varphi_{\ell}\left(\frac{x_{1}}{\sqrt{n}}, \ldots, \frac{x_{n}}{\sqrt{n}} \right))|^\beta)^\alpha$. We can use 
 Beckner's Lemma~\ref{BeL} again. All symmetric functions got replaced by  combination of Hermite polynomials. And then the fact that $\xi_n=\frac{x_1+\dots +x_k}{\sqrt{n}}$ and $\eta_n=\frac{x_{k+1}+\dots+x_n}{\sqrt{n}}$ are uniformly sub-gaussian allows us to make the truncation in $\xi_n$ and then in $\eta_n$ exactly as we did this before. Thus, by the choice of large $T$ and then $S$ expressions $Q$ and $Q'$ can be made as close as possible.

Lemma \ref{InAv} is proved.

\end{proof}

If we denote  $g(w) = \sum_{\ell=0}^{L}a_{\ell}w^{\ell}$ 
then letting $n \to \infty$ and keeping $\frac{k}{n}\to s$ we want to prove the following convergence: 
\begin{theorem}
\label{conv} 
Let $n \to \infty$ and $k$ is such that  $\frac{k}{n}\to s$, $s\in (0, 1)$. Then
\begin{align*}
& \mathbb{E}^{x}_{k}\left(\mathbb{E}^{x}_{n-k} \left|\sum_{\ell=0}^{L} a_{\ell} \varphi_{\ell}\left(\frac{x_{1}}{\sqrt{n}},\ldots, \frac{x_{k}}{\sqrt{n}}, z\frac{x_{k+1}}{\sqrt{n}} \ldots, z\frac{x_{n}}{\sqrt{n}} \right) \right|^{q} \right)^{p/q}\to\\
&\int_{\mathbb{R} }\left( \int_{\mathbb{R}} \left| \int_{\mathbb{R}} \int_{\mathbb{R}} g((u+iv)\sqrt{s} + z (x+iy)\sqrt{1-s}) d\gamma(v) d\gamma(y) \right|^{q} d\gamma(x)\right)^{p/q}d\gamma(u)\,.
\end{align*}
\end{theorem}

Before proving the theorem, notice the following
\begin{corollary}
\label{monot}
Let $g(x) = \sum_{\ell=0}^{L}a_{\ell}x^{\ell}$. Then the map
\begin{align*}
s \to \int_{\mathbb{R} }\left( \int_{\mathbb{R}} \left| \int_{\mathbb{R}} \int_{\mathbb{R}} g((u+iv)\sqrt{s} + z (x+iy)\sqrt{1-s}) d\gamma(v) d\gamma(y) \right|^{q} d\gamma(x)\right)^{p/q}d\gamma(u), \quad 0\leq s \leq 1,
\end{align*}
is monotonically increasing.
\end{corollary}
\begin{proof}
Let $g(x) = \sum_{\ell=0}^{L}a_{\ell}x^{\ell}$, and
\begin{align*}
f_{n}(x_{1}, \ldots, x_{n}) = \sum_{\ell=0}^{L} a_{\ell} \varphi_{n, \ell}\left(\frac{x_{1}}{\sqrt{n}}, \ldots, \frac{x_{n}}{\sqrt{n}} \right)\,.
\end{align*}

\medskip

Choose $s_1<s_2$, put $k_1=[s_1 n], \,k_2=[s_2n]$.  By \cite{B}  for every $n$ we have  the monotonicity (\ref{izrdeba}):
\begin{equation}
\label{monotB}
\mathbb{E}^{k_1} (\mathbb{E}_{n-k_1} | T_{z}^{k_1} f_{n}|^{q})^{p/q} \le \mathbb{E}^{k_2} (\mathbb{E}_{n-k_2} | T_{z}^{k_2} f_{n}|^{q})^{p/q}\,.
\end{equation}
 Theorem~\ref{conv} and Lemma~\ref{InAv}  claim that the limits of expressions ($i=1,2$),  $\cL_i(n):=\mathbb{E}^{k_i} (\mathbb{E}_{n-k_i} | T_{z}^{k_i} f_{n}|^{q})^{p/q}$
exist, and thereby by (\ref{monotB}) we have 
$$
\lim_{n\to\infty} \cL_1(n) \le \lim_{n\to\infty} \cL_2(n)\,.
$$
 Moreover, Theorem \ref{conv} gives these limits of $\cL_i(n)$ as
$$
\int_{\mathbb{R} }\left( \int_{\mathbb{R}} \left| \int_{\mathbb{R}} \int_{\mathbb{R}} g((u+iv)\sqrt{s_i} + z (x+iy)\sqrt{1-s_i}) d\gamma(v) d\gamma(y) \right|^{q} d\gamma(x)\right)^{p/q}d\gamma(u),\, i=1, 2\,.
$$
The corollary is proved.
\end{proof}

\medskip

Now we are going to prove Theorem \ref{conv}.

\begin{proof}
By Lemma~\ref{InAv} it is enough to show that 
\begin{align*}
&  \mathbb{E}^{x}_{k}\left(\mathbb{E}^{x}_{n-k} \left|\mathbb{E}^{y}\sum_{\ell=0}^{L}a_{\ell} \left(\frac{x_{1}+\ldots+x_{k}}{\sqrt{n}}+i \frac{y_{1}+\ldots+y_{k}}{\sqrt{n}}+z \left[\frac{x_{k+1}+\ldots+x_{n}}{\sqrt{n}}+i \frac{y_{k+1}+\ldots+y_{n}}{\sqrt{n}} \right]\right)^{\ell} \right|^{q} \right)^{p/q}\to\\
&\int_{\mathbb{R} }\left( \int_{\mathbb{R}} \left| \int_{\mathbb{R}} \int_{\mathbb{R}} g((u+iv)\sqrt{s} + z (x+iy)\sqrt{1-s}) d\gamma(v) d\gamma(y) \right|^{q} d\gamma(x)\right)^{p/q}d\gamma(u)\,.
\end{align*}
We already saw from \eqref{cut2} that
\begin{align}
\label{cut3}
&\bE_\xi (\bE_\eta |\bE_{\zeta, \tau} g(\xi_n+ i \zeta_n + z(\eta_n+i \tau_n))|^{\beta})^\alpha =\\
& \bE_\xi {\bf 1}_{|\xi_n|\le T}(\bE_\eta | {\bf 1}_{|\eta_n|\le S(T)}\bE_{\zeta, \tau} g(\xi_n+ i \zeta_n + z(\eta_n+i \tau_n))|^{\beta})^\alpha + \varepsilon_T+Ae^{BT} \delta_S\,,
\end{align}

But we can continue truncating the variables. Now it is the turn of $\zeta_n$ and $\tau_n$. Using that $\zeta_n$ and $\tau_n$ are uniformly sub-gaussian we can write
\begin{align*}
&\bE_\xi(\bE_\eta |\bE_{\zeta, \tau} g(\xi_n+ i \zeta_n + z(\eta_n+i \tau_n))|^{\beta})^\alpha = \\
&\bE_\xi {\bf 1}_{|\xi_n|\le T}(\bE_\eta {\bf 1}_{|\eta_n|\le S} |\bE_{\zeta, \tau}{\bf 1}_{|\zeta_n|\le R(T,S), |\tau_n|\le P(T, S, R)} g(\xi_n+ i \zeta_n + z(\eta_n+i \tau_n))|^{\beta})^\alpha +\\
&(\eps_T+Ae^{BT}\delta_S + Ae^{B(T+S)}\gamma_R +Ae^{B(T+S+R)}\nu_P),
\end{align*}
where $\eps_T$, $\delta_S$, $\gamma_R$ and $\nu_P$ can be chosen to be small  in such order to  make $\eps_T+Ae^{BT}\delta_S + Ae^{B(T+S)}\gamma_R +Ae^{B(T+S+R)}\nu_P$ as small as we wish.

\medskip

Let $\mu_n$ be the density of distribution of $\zeta_n/\sqrt{s}$, and  $\mu'_n$ be the density of distribution of $\tau_n/\sqrt{1-s}$. Consider the Lipschitz function   $G(x_1, x_2, y_1, y_2) = g((x_1+ iy_1)\sqrt{s} + z\sqrt{1-s}(x_2+iy_2))$ defined on $[-T,T]\times[-S,S]\times[-P,P]\times[-R,R]$. It is clear that for sufficiently large $n$ we have 
\begin{align*}
\int_{-P}^{P}G(x_{1},x_{2},y_{1},y_{2}) d\mu_{n}(y_{1}) = \int_{-P}^{P}G(x_{1},x_{2},y_{1},y_{2})  d\gamma(y_{1}) + \varepsilon_{n}
\end{align*}
where $\varepsilon_{n}$ is small uniformly for all $(x_{2},y_{1},y_{2}) \in[-S,S]\times[-P,P]\times[-R,R]$. Iterating this observation we easily obtain that 
\begin{align*}
&\lim_{n \to \infty} \int_{-T}^T \left(\int_{-S}^S \left|\int_{-R}^R\int_{-P}^P G(x_1, x_2, y_1, y_2) d\mu_n(y_1)d\mu'_n(y_2)\right|^\beta d\mu'_n(x_2)\right)^\alpha d\mu_n(x_1) = \\
&\int_{-T}^T \left(\int_{-S}^S \left|\int_{-R}^R\int_{-P}^P G(x_1, x_2, y_1, y_2) d\gamma(y_1)d\gamma(y_2)\right|^\beta d\gamma(x_2)\right)^\alpha d\gamma(x_1)
\end{align*}
 This finishes the proof of the theorem. 
\end{proof}

\section{Concluding remarks: Hausdorff--Young inequality and Beckner--Janson flow}
\subsection{Hausdorff--Young inequality}
Let $1<p\leq 2$ and $q :=\frac{p}{p-1}$ be the conjugate exponent. We define the Fourier transform as  
\begin{align*}
\widehat{f}(x) = \int_{\mathbb{R}} f(y) e^{-2 \pi i x \cdot y} dy.
\end{align*}
Take any compactly supported $f$, and define the following map 
\begin{align}\label{map2}
\varphi(s) \stackrel{\mathrm{def}}{=} \left[\int_{\mathbb{R}} \left( \int_{\mathbb{R}} \left|\widehat{f(y) e^{-\frac{sy^{2}}{2p(1-s)}}} \right|^{q}\left(\frac{u}{2 \pi i p (1-s)} + x \right) dx\right)^{p/q}   \frac{e^{-\frac{u^{2}}{2s(1-s)}}}{\sqrt{2 \pi s (1-s)}}du\right]^{1/p}\quad \text{for} \quad 0<s<1.
\end{align}
We will explain that the monotonicity of (\ref{mon2}) for $z=i\sqrt{p-1}$ immediately implies that $\varphi(s)$ is increasing    
on the interval $(0,1)$, and also we will see that  
\begin{align}
&\varphi(0) = \| \widehat{f}\|_{L^{q}}, \label{lim1}\\
&\varphi(1) =  \sqrt{\frac{p^{1/p}}{q^{1/q}}} \cdot  \| f\|_{L^{p}}. 
\end{align}
These  conditions immediately provide us with the celebrated result of W.~Beckner \cite{B}, namely the Hausdorff--Young inequality with sharp constants
\begin{align*}
\|\widehat{f}\|_{L^{q}(\mathbb{R})} \leq \sqrt{\frac{p^{1/p}}{q^{1/q}}}\cdot  \| f\|_{L^{p}(\mathbb{R})}.
\end{align*}

Indeed, let 
\begin{align*}
g(x) = \sum_{\ell =0}^{N} a_{\ell} x^{\ell} \quad\text{and} \quad  \tilde{g}(x) = \sum_{\ell =0}^{N} a_{\ell} H_{\ell}(x),
\end{align*}
where 
\begin{align}\label{herm}
H_{\ell}(x) = \int_{\mathbb{R}} (x+iy)^{\ell} d\gamma(y)
\end{align}
 is the Hermite polynomial of degree $\ell$. 
Then notice that 
\begin{align}\label{kargi}
 \int_{\mathbb{R}}\int_{\mathbb{R}}g\left( z(x+iy)\sqrt{1-s}+(u+iv)\sqrt{s}\right)d\gamma(y) d\gamma(v) =
 M_{\sqrt{s+(1-s)z^{2}}} \tilde{g} \left(\frac{u\sqrt{s}+zx\sqrt{1-s}}{\sqrt{s+(1-s)z^{2}}} \right)
\end{align}
where $M_{w}f(x)$ is the Mehler semigroup 
\begin{align}\label{meh}
M_{w}\tilde{g}(x) =\sum_{\ell =0}^{N}  a_{\ell} w^{\ell} H_{\ell}(x)= \int_{\mathbb{R}} \tilde{g}(y) \frac{\exp\left(-\frac{(xw-y)^{2}}{2(1-w^{2})} \right)}{\sqrt{2 \pi (1-w^{2})}} dy, \quad |w| \leq 1\, \quad w \in \mathbb{C}. 
\end{align}
Indeed, by linearity this is enough to check only for $g(x)=x^{m}$, and use (\ref{herm}), (\ref{meh})  and  the identity 
\begin{align*}
\int_{\mathbb{R}} \int_{\mathbb{R}} P(z_{1} u + z_{2} v) d\gamma(u) d\gamma(v) = \int_{\mathbb{R}}P\left( x \sqrt{z_{1}^{2}+z_{2}^{2}}\right) d\gamma(x)
\end{align*}
for any complex polynomials $P(x)$ and any  $z_{1}, z_{2} \in \mathbb{C}$. Finally notice the relation between $M_{z} h(x)$ and the Fourier transform 
\begin{align}\label{fur2}
M_{w} h(x) = \frac{e^{-\frac{x^{2}w^{2}}{2(1-w^{2})}}}{\sqrt{2\pi (1-w^{2})}}\cdot \left(\widehat{he^{-\frac{y^{2}}{2(1-w^{2})}}} \right)\left(-\frac{xw}{2 \pi i(1-w^{2})}\right)
\end{align}
Therefore 
\begin{align}\label{ori}
M_{\sqrt{s+(1-s)z^{2}}} \tilde{g} \left(\frac{u\sqrt{s}+zx\sqrt{1-s}}{\sqrt{s+(1+s)z^{2}}} \right) =\frac{e^{-\frac{(u\sqrt{s}+zx\sqrt{1-s})^{2}}{2(1-s)(1-z^{2})}}}{\sqrt{2 \pi (1-s)(1-z^{2})}}\left( \widehat{ \tilde{g}e^{-\frac{y^{2}}{2(1-s)(1-z^{2})}}}\right)\left(-\frac{u\sqrt{s} + zx\sqrt{1-s}}{2 \pi i (1-s) (1-z^{2})}\right).
\end{align}

Finally taking $z=i\sqrt{p-1}$, 
\begin{align}\label{perexod} 
\tilde{g}(y) =f(y) e^{\frac{y^{2}}{2 p}} (2 \pi)^{1/2p}, 
\end{align}
  and combining (\ref{kargi}) and (\ref{ori})  we obtain 
\begin{align}
&\int_{\mathbb{R} }\left( \int_{\mathbb{R}} \left| \int_{\mathbb{R}} \int_{\mathbb{R}} g((u+iv)\sqrt{s} + z (x+iy)\sqrt{1-s}) d\gamma(v) d\gamma(y) \right|^{q} d\gamma(x)\right)^{p/q}d\gamma(u)=\nonumber\\
&\int_{\mathbb{R}}\left( \int_{\mathbb{R}}\left| M_{\sqrt{s+(1-s)z^{2}}} \tilde{g} \left(\frac{u\sqrt{s}+zx\sqrt{1-s}}{\sqrt{s+(1-s)z^{2}}} \right) \right|^{q} d\gamma(x) \right)^{p/q} d\gamma(u)\;\stackrel{(z=i\sqrt{p-1})}{=}\;(\varphi(s))^{p} \cdot \frac{q^{p/2q}}{p^{1/2}}. \label{bolo}
\end{align} 

Clearly 
\begin{align}
&\varphi(0)  \stackrel{(\ref{bolo})}{=}  \frac{p^{1/2p}}{q^{1/2q}} \|M_{i\sqrt{p-1}}\tilde{g}\|_{L^{q}(d\gamma)}\stackrel{(\ref{fur2}), (\ref{perexod})}{=}\|\widehat{f} \|_{q}; \label{limm1} \\
&\varphi(1) \stackrel{(\ref{bolo})}{=}  \frac{p^{1/2p}}{q^{1/2q}} \|M_{1} \tilde{g}\|_{L^{p}(d\gamma)} =\frac{p^{1/2p}}{q^{1/2q}} \|\tilde{g}\|_{L^{p}(d\gamma)} \stackrel{(\ref{perexod})}{=}\frac{p^{1/2p}}{q^{1/2q}}  \|f\|_{p}.\label{limm2}
\end{align}

\vskip1cm

\subsection{Beckner--Janson flow} Finally we would like to mention that the left hand side of (\ref{bolo}) for an arbitrary $|z|<1$ can be written as a composition of 3 heat flows. 
Consider the classical heat semigroup $P_{s} h$ with $\partial_{s} P_{s} h = \frac{\Delta}{2} P_{s} h$ and $P_{0}h = h$. We analytically extend the definition of $P_{s} h(x)$ to complex numbers  $s$ and $x$  as follows 
\begin{align*}
P_{s} h (x) = \int_{\mathbb{R}}h(t) \frac{e^{-\frac{(x-t)^{2}}{2s}}}{\sqrt{2 \pi s}}dt.
\end{align*}
Notice that in the exponent we have $(x-t)^{2}$ but not $|x-t|^{2}$ so that the extension is indeed analytic. Also $\sqrt{z}$ we understand in the sense of principal branch. When the test function $h$ has several variables, say $H(x+y)$ we will write $P_{s}^{y}$ to indicate in which variable we make the heat extension. After the direct calculation we obtain 
\begin{align}\label{magari}
\text{LHS of}\,  (\ref{bolo}) =    P_{s}^{u} \left( P_{1-s}^{x} |P_{(1-s)(1-z^{2})}\tilde{g}(u+zx)|^{q}(0) \right)^{p/q}(0).
\end{align}
To make sure the reader understands the notation let us explain one particular expression.  For example $P_{1-s}^{x} |P_{(1-s)(1-z^{2})}\tilde{g}(u+zx)|^{q}(0)$ means that we take the heat extension of $\tilde{g}$ at time $(1-s)(1-z^{2})$ and consider it at point $(u+zx)$. Then we take absolute value and rise it to the power $q$, and take the heat extension of the result in variable $x$, at time $1-s$ and at point $0$. 

Equality (\ref{magari}) follows from (\ref{bolo}) and the identity  
\begin{align}
P_{(1-s)(1-z^{2})}\tilde{g}(u+zx) = M_{\sqrt{s+(1-s)z^{2}}} \tilde{g} \left(\frac{u+zx}{\sqrt{s+(1-s)z^{2}}} \right) 
\end{align}
where $z$ can be arbitrary $|z|< 1$, $z \in \mathbb{C}$. By direct differentiation in $s$  it was checked in \cite{SJ, Hu} that the  mapping 
$$
s \mapsto P_{s}^{u} \left( P_{1-s}^{x} |P_{(1-s)(1-z^{2})}\tilde{g}(u+zx)|^{q}(0) \right)^{p/q}(0)
$$  
is increasing on $[0,1]$. 

\subsection{Another way to see that monotonicity (\ref{mon2}) implies Hausdorff--Young inequality}
\label{another}

Let 
$$
A_\zeta(x):= e^{\zeta x-\zeta^2/2}\,.
$$
\begin{lemma}
\label{A}
\begin{align*}
\frac1{\sqrt{2\pi}}\int_{\mathbb{R}} e^{\zeta (x+iy)} e^{-y^2/2} dy = A_\zeta(x)\,.
\end{align*}
\end{lemma}
\begin{proof}
Direct calculation.
\end{proof}

Take a test function 
$$
g(w) = \sum_{\ell=1}^L c_\ell e^{ t_\ell w },
$$
and consider ($s\in [0,1]$)
$$
\Phi_s(x, u) := \int_{\mathbb{R}}\int_{\mathbb{R}} g((x+iy)\sqrt{s} + z(u+iv)\sqrt{1-s}) d \gamma(y) d\gamma(v)\,.
$$
Then by Lemma \ref{A}
\begin{align*}
&\Phi_s(x, u) = \sum c_\ell A_{t_\ell \sqrt{s}}(x) A_{t_\ell z\sqrt{1-s}}(u),\\
&\Phi_1(\sqrt{2\pi p}x, u) = \sum c_\ell A_{t_\ell} (\sqrt{2\pi p} x)=\sum c_\ell e^{t_\ell \sqrt{2\pi p} x -t_\ell^2/2}, \\
& \Phi_0(x, \sqrt{2\pi q}u) = \sum c_\ell A_{z t_\ell} (\sqrt{2\pi q} u)=\sum c_\ell e^{i\sqrt{2\pi p} t_\ell u-\left(it_\ell\sqrt{\frac{p}{q}}\right)^2/2}\,.
\end{align*}
In the last inequality we substitute
$$
z=i\sqrt{p/q}\,.
$$
The connection with the Fourier transform is given by
\begin{lemma}
\label{F}
$$
\cF(e^{t_\ell \sqrt{2\pi p} x -t_\ell^2/2} e^{-\pi x^2}) = e^{-i\sqrt{2\pi p} t_\ell u-\left(it_\ell\sqrt{\frac{p}{q}}\right)^2/2}e^{-\pi u^2}\,.
$$
\end{lemma}
\begin{proof}
Direct calculation.
\end{proof}

Set
$$
\phi(s) := \frac1{\sqrt{2\pi}}\left(\int_{\mathbb{R}} \frac1{\sqrt{2\pi}}\int_{\mathbb{R}} |\Phi_s(x', u')|^q e^{-\frac{(u')^2}{2}} du'\right)^{p/q}e^{-\frac{(x')^2}{2}} dx'\,.
$$
Then the monotonicity \eqref{mon2} implies
\begin{equation}
\label{mon3}
\phi(0)\le \phi(1)\,.
\end{equation}

On the other hand, making change of variable $x'=\sqrt{2\pi p} x$ in $\phi(1)$ and $u'=-\sqrt{2\pi q} u$ in $\phi(0)$, we get
\begin{align*}
&\phi(0) = (\sqrt{q})^{p/q}\left( \int_{\mathbb{R}}  e^{-\pi u^2q} \left| \sum c_\ell  e^{-i\sqrt{2\pi p} t_\ell u-\left(it_\ell\sqrt{\frac{p}{q}}\right)^2/2}\right|^q du\right)^{p/q}=\\
 &(\sqrt{q})^{p/q} \left(\int_{\mathbb{R}}  \left|  e^{-\pi u^2 }\sum c_\ell  e^{-i\sqrt{2\pi q} t_\ell u-\left(it_\ell\sqrt{\frac{p}{q}}\right)^2/2}\right|^{q} du\right)^{p/q} = (\sqrt{q})^{p/q}\left(\int_{\mathbb{R}} \left| \cF(e^{-\pi x^2 }\sum c_\ell e^{t_\ell \sqrt{2\pi p} x -t_\ell^2/2} )(u)\right|^{q}du\right)^{p/q}
\end{align*}
by Lemma \ref{F}. At the same time
$
\phi(1)= \sqrt{p} \int_{\mathbb{R}} \left| e^{-\pi x^2 }\sum c_\ell e^{t_\ell \sqrt{2\pi p} x -t_\ell^2/2}\right|^p dx.
$
Comparing two last lines  and using \eqref{mon3} we get
$$
\|\cF(e^{-\pi x^2 }\sum c_\ell e^{t_\ell \sqrt{2\pi p} x -t_\ell^2/2} )\|_{q} \le \frac{\sqrt{p}^{1/p}}{\sqrt{q}^{1/q}} \|e^{-\pi x^2 }\sum c_\ell e^{t_\ell \sqrt{2\pi p} x -t_\ell^2/2}\|_p\,.
$$


\begin{thebibliography}{999}

\bibitem{B}{\sc W.~Beckner}, {\em Inequalities in Fourier analysis}, Ann. of Math., vol. 102, No.1,  (1975), pp. 159--182.

\bibitem{Hu} {\sc Y.~Hu}, {\em Analysis on Gaussian Spaces}, World Scientific, 2016.

\bibitem{SJ} {\sc S.~Janson}, {\em On complex hypercontractivity}. J. Funct. Anal. 151 (1997), no. 1,
270--280.

\bibitem{Weissler} {\sc F.~Weissler}, {\em Two-point inequalities, the Hermite semigroup, and the Gauss--Weierstrass semigroup}, J.~Funct. Anal. Vol. 31, (1979), 102--121.


\end{thebibliography}
\end{document}